\documentclass[11pt,a4paper,twoside]{amsart}
\usepackage[square,sort,comma,numbers]{natbib}
\usepackage{hyperref}
\usepackage{amssymb}
\usepackage{amsfonts}
\usepackage{mathrsfs}
\usepackage{amsmath,amscd}
\usepackage{amssymb}
\usepackage{amsthm}
\usepackage{enumerate}
\usepackage[english]{babel}
\usepackage{hyperref}

\usepackage{tikz}
\usetikzlibrary{calc}
\usetikzlibrary{fadings,arrows}

\usepackage[all]{xy}

\usepackage{t1enc}

\newtheorem{theorem}{Theorem}[section]
\newtheorem{lemma}[theorem]{Lemma}
\newtheorem{proposition}[theorem]{Proposition}                            
\newtheorem{corollary}[theorem]{Corollary}
\newtheorem{definition}[theorem]{Definition}
\newtheorem{remark}{\it Remark\/}
\newtheorem{example}{\it Example\/}

\def\commutatif{\circlearrowleft}

\newcommand{\s}[0]{\ensuremath{P}}

\newcommand{\R}[0]{\ensuremath{R}}
\newcommand{\Q}[0]{\ensuremath{Q}}

\newcommand{\RR}[0]{\ensuremath{\mathbb{R}}}
\newcommand{\AF}[0]{\ensuremath{\mathbb{A}}}
\newcommand{\ZZ}[0]{\ensuremath{\mathbb{Z}}}

\newcommand{\KK}[0]{\ensuremath{\mathbb{K}}}
\newcommand{\CC}[0]{\ensuremath{\mathbb{C}}}

\newcommand{\EN}[1]{\ensuremath{\operatorname{\Gamma^{\star}(#1)}}}
\newcommand{\REN}[1]{\ensuremath{\operatorname{\Sigma_{#1}}}}
\newcommand{\PN}[1]{\ensuremath{\operatorname{\Gamma_{+}(#1)}}}
\newcommand{\FN}[1]{\ensuremath{\operatorname{\Gamma(#1)}}}

\newcommand{\spec}[0]{\ensuremath{\operatorname{Spec}}}
\newcommand{\Id}[0]{\ensuremath{\operatorname{Id}}}
\newcommand{\specrel}[0]{\ensuremath{\operatorname{\bf Spec}}}

\newcommand{\sing}[0]{\ensuremath{\operatorname{Sing}}}
\newcommand{\Ess}[0]{\ensuremath{\operatorname{Ess}}}
\newcommand{\Hom}[0]{\ensuremath{\operatorname{Hom}}}

\newcommand{\ord}[0]{\ensuremath{\operatorname{Ord}}}

\newcommand{\p}[0]{\ensuremath{\operatorname{p}}}
\newcommand{\pinf}[2]{\ensuremath{\operatorname{p}_{#1}^{#2}}}
\newcommand{\f}[0]{\ensuremath{\operatorname{f}}}
\newcommand{\F}[0]{\ensuremath{\operatorname{F}}}

\newcommand{\g}[0]{\ensuremath{\operatorname{g}}}

\newcommand{\sym}[0]{\ensuremath{\operatorname{Sym}}}
\newcommand{\E}[0]{\ensuremath{\operatorname{E}}}

\newcommand{\ens}{\mathcal{E}ns}
\newcommand{\salg}{\s\!-\!\mathcal{A}lg}

\newcommand{\HS}[1]{\operatorname{HS}^{#1}_{\R\slash\s}}
\newcommand{\mJet}[3]{\operatorname{#1(#2)}_{#3}}
\newcommand{\arcSpace}[2]{\operatorname{#1(#2)}_{\infty}}
\newcommand{\mjp}[1]{\operatorname{\p}_{#1}}

\newcommand{\Arc}[2]{\operatorname{#1(#2)}_{\infty}}
\newcommand{\HSF}[3]{\operatorname{HS}^{#1}_{#2 \slash #3}}
\newcommand{\DerR}[2]{\operatorname{Der}^{#1}_{\s}(\R,#2)}

\newcommand{\red}[1]{#1_{red}}

\newcommand{\h}[0]{\ensuremath{\operatorname{h}}}

\newcommand{\hf}[2]{\h_{#1}(#2)}

 \usetikzlibrary{arrows,chains,matrix,positioning,scopes}
\begin{document}

\title[Deformation of Spaces of m-jets]{Deforming Spaces of m-jets of  Hypersurfaces Singularities}


\author{Maximiliano Leyton-\'Alvarez}

\address{Instituto Matem\'atica y F\'isica, Universidad de Talca,
Camino Lircay S$\backslash$N, Campus Norte, Talca, Chile.}

\email{leyton@inst-mat.utalca.cl}

\date{\today}

\thanks{Partially supported by  FONDECYT $\rm{N}^\circ: 1170743$ }

\begin{abstract}
\selectlanguage{english}
\noindent {\bf  Deformation  of  Spaces of m-jets.}  
Let $\KK$ be an algebraically closed field of characteristic zero,
and $V$ a hypersurface defined by an irreducible polynomial $\f$ with coefficients in $\KK$. 

In this article we prove that an Embedded Deformation of $V$ which admits a Simultaneous
Embedded Resolution induces, under certain  conditions, a deformation of the the space of $m$-jets $V_m$, $m\geq 0$.
An example of an Embedded Deformation of $V$ which admits a Simultaneous
Embedded Resolution is a $\Gamma(\f)$-deformation of $V$, where $V$ has at most one 
isolated singularity, and $\f$ is non degenerate with respect to the Newton Boundary $\Gamma(\f)$.

\end{abstract}

\maketitle

\section{Introduction}

Let $V$ be an algebraic variety over a field $\KK$, and $m$ a positive integer. 
Intuitively the Space of $m$-jets, $V_m$, (resp. Space of arcs, $V_{\infty}$,) 
is the set of morphisms

\begin{center}
 $\spec \KK[t]/(t^{m+1})\rightarrow V$ (resp.  $\spec \KK[[t]] \rightarrow V$)
\end{center}
equipped with a ``natural''  structure of $\KK$-schema.  During the late $60s$
Nash studied these spaces as a tool to understand the local geometry of the singular locus   of $V$. 
He was specifically interested in understanding the {\it Essential Divisors over} $V$ (see Section \ref{de:ess-div}).
He constructed an injective application, known as the {\it Nash Application},  from the set of { \it Nash Components of} 
$V$ (see Section \ref{de:Nas-comp}) to the set of Essential Divisors. 
Then the question raised was:  Is the Nash application bijective?\\

Obtaining a complete answer to this question took many decades.  
In the year $2003$ Ishii and Kollar (see \cite{IsKo03}) showed the first example of a 
variety such that its Nash application is not surjective.  
It is worth mentioning that this example is a hypersurface of dimension $4$.  
In the year $2012$ Fernandez de Bobadilla and Pe-Pereira (see \cite{BoPe12a}) gave an affirmative
answer to the question for the case of surfaces over $\CC$.  Finally between the years $2012$ and $2013$ 
Johnson, Kollar, and de Fernex constructed examples of  hypersurfaces of dimension $3$ where the Nash 
application is not surjective (see  \cite{dFe13},  \cite{JoKo13}  and \cite{Kol12}).  
It is important to mention that many mathematicians worked hard on this problem, making valuable progress with respect to the problem.  
For example:   \cite{dFDo14}, \cite{Bob12}, \cite{Gon07}, \cite{Gole97}, \cite{IsKo03}, \cite{Ish05},\cite{Ish06}, 
\cite{Ley11a},  \cite{Ley14}, \cite{Lej80}, \cite{LeRe99}, \cite{Mor08}, \cite{Pet09}, \cite{Ple08}, \cite{PlPo06}, \cite{PlPo08}, \cite{PlSp12},
\cite{Reg95} etc.  
Unfortunately it is not possible to comment on and cite all the existing works.\\

However, despite all the progress made, giving an exact description 
of the image of the Nash application remains an open problem.\\

Informally speaking, hidden within the  Spaces of $m$-jets  and the Space of arcs 
lies much information on geometry of the subjacent variety.  For example:

In the year 1995, Kontsevich, using these spaces, introduced the {\it motivic integration} to resolve the Batyrev 
conjecture on the Calabi-Yau varieties (see \cite{Kon95}). 
For more references on motivic integration see \cite{Bli11},  \cite{DeLo99}, \cite{DelO02} and \cite{Loo02}.  

Other examples are found in the articles   \cite{EiMu04}, \cite{Mus01}  written by  Ein and Musta\c{t}\u{a} where, 
amongst other things, it is demonstrated that if $V$ is locally a complete intersection, then $V$
only has rational singularities (resp. log-canonical singularities) if and only if $V_m$ is irreducible (resp. equidimensional) 
for all $m\geq 0$.  

Another interesting application for the Spaces of $m$-jets and the Space of arcs is 
obtaining identities and invariants associated to $V$, for example see the articles \cite{BMS13}, and \cite{DeLo99}.\\  

One of the main subjects of \cite{Ley14} was the study of the following problem:  When a ``deformation'' of the 
 variety  $V$ induces a ``natural deformation'' of the Spaces of $m$-jets $V_m$ (for greater precision see Section  \ref{sec-def-prob}).
In general a deformation of $V$ does not induce a deformation of the spaces of $V_m$, see Example \ref{ex:no-def}, 
however there exist important families of examples in which this property is satisfied.  
For example if $V$ is locally a complete intersection, and only has an isolated singularity 
of  log-canonical type (see Proposition \ref{pr:def-mjets-lic}).\\  

The idea to study the Spaces of $m$-jets in families is very natural, 
for example Mourtada in \cite{Mou11} studies families of complex plane branches with constant
topological type, and obtains formulas for the calculation of the number and dimension of the irreducible components 
of the  Spaces of $m$-jets.\\

In \cite{Ley14}  it was obtained that if $V$ is a Pham-Briekson hypersurface defined by an
irreducible polynomial $\f$, and $W\rightarrow \spec \KK[[s]]$ is a  $\Gamma(\f)$-deformation (see Section \ref{se:new-def}) then  $W$ induces a deformation of $\red{(V_m)}$ 
 (for more details see Theorem \ref{th:pham-bri}). In this article we will generalize this result.\\ 

 Let $V\subset \AF_{\KK}^{n+1}$ be a hypersurface, and $S$ is a finite type $\KK$-scheme. 
 In this article we will
consider an Embedded deformation of $V\subset \AF_{\KK}^{n+1}$ over $S$ which admits a 
Simultaneous Embedded 
Resolution (see Section \ref{sec:Hip-Gen}), and we prove that this deformation induces a deformation of the
scheme  $V_m$  (see Theorems \ref{th:flat-mjet} and \ref{th:flat-mjet2}).\\

 An example of an Embedded Deformation which admits a Simultaneous Embedded 
Resolution is the following: Let $V\subset \AF_{\KK}^{n+1}$ be the hypersurface defined 
by an irreducible polynomial $\f$ non degenerate with respect to the Newton boundary $\FN{\f}$.
In addition we will suppose that $V$ has at
most one isolated singularity in the ``origin'' of  $\AF_{\KK}^{n+1}$, 
and that $V$ does not contain orbits
of the torus $T:=(\KK^{\star})^{n+1}$ of dimension greater than or equal to $1$.
So a $\Gamma(\f)$-deformation 
of $V$ admits a Simultaneous Embedded Resolution.\\

It is worth mentioning that under certain conditions the deformations that we will consider have a 
 property to preserve the topological type of $V$ (see \cite{Ger86}, \cite{Kou76}, \cite{OSh87}, \cite{Tei73}, \cite{LDRa76} and   \cite{LDRa76}).\\  
 
The article is organized as follows:\\

	In the section  {\it Preliminaries and Reminders}  we will briefly present the known definitions and results of 
	the Hasse-Schmidt derivations and the Space of $m$-jets and the Space of arcs relatives to a given scheme. 
	We will finish this section with a brief introduction to the Nash problem. 
	The purpose of this section is to make reading the article easier.\\ 

	The main results of this article will be developed in the section {\it Deformation of Spaces of $m$-jets}.
	In Section \ref{sec-def-prob} we will briefly explain the problem we will study.  
	Then, in Section \ref{sec:Hip-Gen}, we will give the general hypothesis 
	of the principal result of this article	and we will show that the  $\Gamma(\f)$-deformations 
	satisfy them.   
	Finally, in Section \ref{sec:def-hyp}, we will prove the mains results of this article. 
	We will finish the section with an application to the Nash problem.

\section*{Acknowledgments}  
I would like to give special thanks to Hussein Mourtada, Peter Petrov, 
Camille Pl\'enat and Mark Spivakovsky for all the fruitful discussions.
I would also like to thank Rachel Rogers for her unconditional support during this time. \\


\section{Preliminaries and Reminders}
\label{se:Pre-Rap}
\subsection{Hasse-Schmidt Derivations}
\label{ssec:h-s}
In this section we give some results and notions based on the Hasse-Schmidt derivations,
the {\it Relative Space of $m$-jets}, and the {\it Relative  Space of arcs}.  Here our goal is not to give an exhaustive 
overview of the theory, but to give an overview of the latter in the context which interests us.  
For more details, see \cite{Voj07}.\\

Let $\KK$ be an algebraically closed field of 
characteristic zero, and $\s$ a $\KK$-algebra and $\R$,$\Q$ two $\s$-algebras.\\

A {\it Hasse-Schmidt derivation of order} $m\in \ZZ_{\geq 0}$ from $\R$ to $\Q$ is an $m+1$-tuple  $(D_0,\cdots, D_m)$
where  $D_0:R\rightarrow Q$ is a homomorphism of $\s$-algebra, 
and  $D_i:R\rightarrow Q$, $1\leq i\leq m$ is a homomorphism of abelian groups, 
which satisfies the following properties:

\begin{enumerate}

 \item[(i)]  $D_i(p)=0$ for all $p\in \s$  and for all $1\leq i\leq m$;
 \item[(ii)]   For all the elements $x$ and  $y$  of $\R$ and for all integers $1\leq k \leq m$, we have:  
 \begin{center}
  $D_k(xy)=\sum_{i+j=k}D_i(x)D_j(y)$.
   \end{center}
 \end{enumerate}
 \vspace{0.1cm}
 
 We remark that if $\phi:Q\rightarrow Q'$ is a homomorphism of $\s$-algebra, and $(D_0,\cdots, D_m)$  a 
 Hasse-Schmidt derivation from $\R$ to $\Q$, then $(\phi\circ D_0,\cdots, \phi \circ D_m)$ is a 
 Hasse Schmidt derivation from $\R$ to $\Q'$.  Consequently we can define the following functor:

  \vspace{0.3cm}
 \begin{center}
  $\DerR{m}{\;\cdot\;}:\salg\rightarrow \ens;$ $\Q\mapsto \DerR{m}{\Q}$
 \end{center}
  \vspace{0.3cm}
   where $\salg$ (resp. $\ens $) is the category of $\s$-algebras (resp. of sets) and $\DerR{m}{\Q}$
  is the the set of Hasse-Schmidt derivations of the order $m$ of $\R$ to $\Q$.\\  
  
 In the same way, we can define the Hasse-Schmidt derivation of order $\infty$, and the functor $\DerR{\infty}{\;\cdot\;}$ by using, in place of  $m\!+\!1$-tuple 
 $(D_0,\cdots, D_m)$, an infinite sequence $D_0,D_1,\cdots $.\\

Let $\HS{m}$  be the $\R$-algebra quotient of the algebra $\sym( \displaystyle\mathop{\oplus}\limits_ {i=1}^{m} \displaystyle\mathop{\oplus}\limits_{ x\in R} \R d^{i}x)$
by the ideal $I$ generated by the union of the following sets:

\begin{itemize}
 \item $\{d^i(x+y)-d^i(x)-d^i(y)\mid x,y\in \R; \; 1\leq  i\leq  m\},$
 \item $\{d^i(s)\mid s \in \s; \; 1\leq i\leq  m\}$
 \item  $\{d^k(xy)-\sum_{i+j=k}d^i(x)d^i(y)\mid x,y\in \R; \; 1\leq k\leq  m\}$.\\
 \end{itemize}

We also remark that $\HS{m} $ 
is a graduated algebra by the graduation $d^ix\mapsto i$.  According to the definition of the algebra $\HS{m}$,
we naturally obtain a sequence of homomorphisms of graduated algebras:  

\vspace{0.3cm}
\begin{center}
 
$R:=\HS{0}\rightarrow \HS{1}\cdots\rightarrow\HS{m}\rightarrow \HS{m+1}\rightarrow \cdots$
\end{center}
\vspace{0.3cm}

We note $\HS{\infty }$ the inductive limit of the inductive system $\HS{m}$, which is to say:  

\vspace{0.3cm}
\begin{center}
 $\HS{\infty }:=\displaystyle\mathop{\lim}\limits_{m\rightarrow \infty}\HS{m}$.
\end{center}
 
\vspace{0.2cm}
\begin{proposition}[\cite{Voj07}] \label{pro:HS-repre}
The Hasse-Schmidt algebra $\HS{m}$,  $0\leq m\leq \infty$, represents the functor of   Hasse-Schmidt derivations  $\DerR{m}{\;\cdot\;}$,
which is to say:

\vspace{0.3cm}
\begin{center}
 $\DerR{m}{\Q}\cong \Hom_{\s}(\HS{m},Q)$,
\end{center}
\vspace{0.3cm}
where $\Q$ is any $\s$-algebra. 
\end{proposition}
The Hasse-Schmidt algebra  $\HS{m}$, $0\leq m\leq \infty$ satisfies the following properties of localization:

\begin{proposition}[\cite{Voj07}]
Let $0\leq m\leq \infty$ and consider  a multiplicative subset $L$ (resp. $L_1$)  of the algebra $\R$ (resp. $\s$).
Then we have:

\begin{enumerate}
 \item $\HS{m}[L^{\;-1}]\cong \HSF{m}{R[L^{\;-1}]}{\s}$;
 \item If the morphism $\s\rightarrow R$ factors through the canonical morphism $\s\rightarrow \s[L_1^{\;-1}]$, 
 then $\HS{m}\cong \HSF{m}{R}{\s[L_1^{\;-1}]}$.
 \end{enumerate}
\end{proposition}

Let $f:X\rightarrow Y$ be a morphism of schemes.  Taking a covering by affine open subsets of $X$ 
compatible with a covering by affine open subsets  of $Y$, and by using the above localization properties, 
we can define a quasi-coherent sheaf, $\HSF{m}{\mathcal O_{X}}{\mathcal{O}_{Y}}$, $1\leq m\leq \infty$, of 
$\mathcal{O}_X$-algebras that does not depend on the choice of the open coverings.

\begin{definition}
 For each $0\leq m\leq \infty$, we place:
 \begin{center}
 $\mJet{X}{Y}{m}:=\specrel \HSF{m}{\mathcal O_{X}}{\mathcal{O}_{Y}}$
 \end{center}
where $\specrel \HSF{m}{\mathcal O_{X}}{\mathcal{O}_{Y}}$ is the relative spectrum of the quasi-coherent sheaf 
of $\mathcal O_{X}$-algebras  $\HSF{m}{\mathcal O_{X}}{\mathcal{O}_{Y}}$.
If $Y:=\spec \KK$, we place $X_{m}:=\mJet{X}{\spec \KK}{m}$.
\vspace{0.3cm}

For $0\leq m < \infty$  (resp. $m=\infty$),  the scheme $\mJet{X}{Y}{m}$ (resp. $\arcSpace{X}{Y}$) 
is called the Space of  $m$-jets  (resp. Space of arcs) of the morphisms $f:X\rightarrow Y$.   
\end{definition}

The following results justify the terminology ``Space of m-jets  and  Space of arcs of the morphism $f:X\rightarrow Y$''  
\vspace{0.2cm}

We remark that we can define the following applications:  

\vspace{0.5cm}
\begin{center}
 $\begin{array}{ccl}
 \DerR{m}{\Q}&\rightarrow & \Hom_{\s}(R,Q[t]/(t^{m+1}));\\
 
 (D_0,...,D_m)&\rightarrow & \left\{\begin{array}{l} 
 R\rightarrow Q[t]/(t^{m+1}); \\ x\mapsto D_0(x)+\cdots+D_m(x)t^m\end{array} \right.
                                                                                                      
\end{array}$
\end{center}
\vspace{0.5cm}

\begin{center}

$\begin{array}{ccl}
 \DerR{\infty}{\Q}&\rightarrow & \Hom_{\s}(R,Q[[t]]);\\
 
 D_0,D_1\cdots &\rightarrow & \left\{\begin{array}{l} 
 R\rightarrow Q[[t]]; \\ x\mapsto D_0(x)+D_1(x)t+\cdots\end{array} \right.
                                                                                                      
\end{array}$
\end{center}
\vspace{0.3cm}

Using the above applications we can demonstrate the following result:

\begin{proposition}[\cite{Voj07}]
\label{pro:fonct-repre-mjet}
 The functor $Q\mapsto \DerR{m}{Q}$, $0\!\leq\!  m\!< \!\infty$, (resp. $Q \mapsto \DerR{\infty}{Q}$)  
 is isomorphic to the functor $Q\mapsto \Hom_{\s}(R,Q[t]/(t^{m+1}))$ 
 (resp.   $Q\mapsto \Hom_{\s}(R,Q[[t]])$).  
\end{proposition}

The following theorems are a consequence of Propositions \ref{pro:HS-repre} and  \ref{pro:fonct-repre-mjet}.

\begin{theorem} \label{th:prop-fonct-mjet}
Let $f:X\rightarrow Y$ be a morphism of schemes.  Then the Space of $m$-jets  of  $f:X\rightarrow Y$
represents the functor of $m$-jets, which is to say

\begin{center}
$\Hom_Y(Z\times_{\KK}\spec \KK[t]/(t^{m+1}), X)\cong \Hom_{Y}(Z,\mJet{X}{Y}{m})$, 
\end{center}
where $Z$ is any $Y$-scheme.  
\end{theorem}

\begin{theorem}\label{th:prop-fonct-arc}  Let $f:X\rightarrow Y$ be a morphism of schemes.  
Then the Space  of arcs of  $f:X\rightarrow Y$   represents the functor of arcs, which is to say

\begin{center}
 $\Hom_Y(Z\widehat{\times}_{\KK}\spec \KK[[t]], X)\cong \Hom_{Y}(Z,\arcSpace{X}{Y})$
\end{center}
where $Z$ is any $Y$-scheme, and $Z\widehat{\times}_{\KK}\spec \KK[[t]]$ is  the formal completion of the scheme 
$Z\times_{\KK}\spec \KK[[t]] $ along the subscheme $Z\times_{\KK}\spec{\KK} $. 
\end{theorem}
The Spaces of $m$-jets  and the Space of arcs of a morphism satisfy  the following  base extension  property:  

\begin{proposition}[\cite{Voj07}]
 \label{pr:mjets-chang-base}
  Let $f:X\rightarrow Y$ and $Y'\rightarrow Y$ be two morphisms of schemes, and we consider $X':=X\times_Y Y'$.  
  Then  $\mJet{X'}{Y'}{m}\cong  \mJet{X}{Y}{m}\times_YY'$,  over $Y'$ for all $0 \leq m\leq \infty$. In 
  particular there exists a commutative diagram:

 \begin{center}
\begin{tabular}{lr}

 $
\xymatrix @!0 @R=1.5pc @C=2pc{  
\mJet{X'}{Y'}{m}\ar@{->}[rrrr] \ar@{->}[dd]&& & 
& \mJet{X}{Y}{m} \ar@{->}[dd]\\
   &&\square&& \\
  Y' \ar@{->}[rrrr] &&& &  Y 
}$   
\end{tabular}
\end{center}
\end{proposition}
The symbol  $\square$  of the diagram is in order to explicitly indicate that the 
commutative diagram is a {\it Cartesian square}.  
\subsection{The Space of arcs and the Nash Problem.}

\label{ssec:nash} 
 In this section we will give the basic definitions of the Nash problem.\\

Let $\KK$ be an algebraically closed field of characteristic zero, $V$ a normal 
algebraic variety over $\KK$, and $\pi:X\rightarrow V$ 
a divisorial resolution of $V$.\\

\subsubsection{The Nash Components}
\label{de:Nas-comp}

We recall that $V_{\infty}$ is the Arcs Space  over  $V$, and that the $K$-points of $V_{\infty}$
are in bijective correspondence with the $K$-arcs over  $V$ where  $K$ is a field extension of $\KK$. 
Let   $\KK_{\alpha}$  be the residual field of the point $\alpha\in V_{\infty}$.  By abuse of notation we also note $\alpha$ the $\KK_{\alpha}$-arc
which corresponds to the point $\alpha$. 
Let $\p:V_{\infty}\rightarrow V$  be the canonical projection $\alpha\mapsto \alpha(0)$ where $0$ is the closed point of $\spec \KK_{\alpha}[[t]]$.
We remark that $\p$ is a morphism.
 The irreducible components of $V_{\infty}^{s}:= \p^{-1}(\sing V)$ are called  the {\it Nash Components of} $V$. 
We note $\mathcal{CN}(V)$ the set of   Nash  Components  of $V$.

\subsubsection{The Essential Divisors}
\label{de:ess-div}

Given any desingularization $\pi':X'\rightarrow V$, the birational application  $(\pi')^{-1}\circ\pi:X\dashrightarrow X'$
is well defined in codimension $1$ ($X$ is a normal variety). 
If $\E$ is an irreducible component of the exceptional fiber of $\pi$, then there exists an open subset $\E^{0}$ of $\E$ over which the application $(\pi')^{-1}\circ\pi$ 
is well defined.

 The divisor $\E$ is called {\it Essential Divisor over} $V$ if for all desingularization $\pi'$ the adherence 
$\overline{(\pi')^{-1}\circ\pi(\E^0)}$ is an irreducible component of $(\pi')^{-1}(\sing V)$,  where $\sing V$  is the singular locus of $V$. 
We note that $\Ess(V)$ the set of essential divisors over $V$.

\subsubsection{The Nash Application}
\label{de:nash-appl}

Given $C\in \mathcal{CN}(V)$, let  $\alpha_{C}$ be the generic point of $C$. 
Nash showed that the application $\mathcal{N}_{V}:\mathcal{CN}(V)\rightarrow \Ess(V)$
which associates to  $C\in \mathcal{CN}(V)$  the adherence  $\overline{\{\widehat{\alpha}_{C}(0)\}}$
is a well defined injective  application (see \cite{Nas95}),
where $\widehat{\alpha}_{C}$  is the lifting to $X$ of the generic point $\alpha_{C}$, which is to say $\pi\circ\widehat{\alpha}_{C}=\alpha_{C}$. 
The Nash problem consists of the study of the surjectivity of $\mathcal{N}_{V}$.\\

In the case of surfaces, the problem remained open until the year $2012$, when Fernandez de Bobadilla and Pe-Pereira
in the article \cite{BoPe12a} showed that the Nash application is bijective for all singularities of surfaces over $\CC$.  
 In the year 2003, Ishii and Kollar   discovered the first example of a
$V$ variety such that the Nash Application $\mathcal{N}_V$ is not bijective; this variety is a 
hypersurface of $\AF^5_{\KK}$ having a unique isolated singularity (see  \cite{IsKo03}).  The examples of varieties of dimension three where 
the Nash application is not bijective appeared during the years $2012$ and $2013$ (see the articles \cite{Kol12}, \cite{dFe13}); 
these examples are hypersurfaces of $\AF^4_{\KK}$ having a unique isolated singularity.  We can find more examples, 
for dimensions greater than or equal to three, in the article \cite{JoKo13}.

\section{ Deformation  of Spaces of m-jets}
\subsection{Definition of the main problem}
\label{sec-def-prob}

 Let $\KK$ be an algebraically closed field of characteristic zero, $V$ a variety over $\KK$, $S$ a $\KK$-scheme, and $0\in S$ a closed point. 
Let us consider a deformation $W$ of $V$ over $S$, which is to say a commutative diagram (Cartesian square):

 \begin{center}
\begin{tabular}{lr}

 $
\xymatrix @!0 @R=1.5pc @C=2pc{  
V\ar@{^{(}->}[rr] \ar@{->}[dd] & 
& W \ar@{->}[dd]^{\varrho}\\
   &\square& \\
  0 \ar@{^{(}->}[rr] & &  S
}$   
\end{tabular}  
\end{center} where $\varrho:W\rightarrow S$ is flat, and $V\cong W\times_S\{0\}$.  Using Proposition \ref{pr:mjets-chang-base}
we obtain the following commutative diagram:

\begin{center}
\begin{tabular}{lr}

 $
\xymatrix @!0 @R=1.5pc @C=2pc{  
V_m\ar@{^{(}->}[rr] \ar@{->}[dd] & 
& \mJet{W}{S}{m} \ar@{->}[dd]^{\varrho_m}\\
   &\square& \\
  0 \ar@{^{(}->}[rr] & &  S
}$   
\end{tabular}
\end{center}
The question we ask ourselves is: when is this diagram a deformation?  In other words, When 
is the morphism $\varrho_m:\mJet{W}{S}{m}\rightarrow S$ flat?\\

In general the morphism $\varrho_m$  is not flat.  

\begin{example}
\label{ex:no-def}

Let $V\subset \AF_{\CC}^3$ be the hypersurface give by the equation $x^4+y^4+z^4=0$, and $W$ the hypersurface of $\AF_{\CC}^3\times \AF_{\CC, 0}$
given by the equation $x^4+y^4+z^4+s=0$.  It is not difficult to verify that the morphism 
$W\rightarrow \AF_{\CC, 0}; (x,y,z,s)\mapsto s$ is a deformation of $V$. 

Let  $p\in \AF_{\CC,0}-\{0\}$, and $W_p:=W\times_S\{p\}$.  
	The fiber $W_p$  is a smooth variety, then the Space of $3$-jets $(W_p)_3$ is irreducible. 
	But the space of $3$-jets $V_3 $ is not equidimensional, which implies that $\mJet{W}{S}{3}\rightarrow S$ is not flat.  
\end{example}

The following result is a direct consequence of the results of
Ein and  Musta\c{t}\u{a} stated in the introduction, and the local criterion for flatness.  For more details see \cite{Ley14}:

\begin{proposition}
\label{pr:def-mjets-lic}
Let $V$ be  a locally  complete intersection variety. 
We suppose that $V$ has at most one isolated singularity of  log-canonical type.  Then the morphism $\mJet{W}{S}{m} \rightarrow S$ 
is a deformation of  $V_m$ for all $0 \leq m\leq \infty$.  
\end{proposition}

In this article we will generalize the following result.\\

Let $V\subset \AF_{\KK}^n$ be the hypersurface given  by polynomial  $\f \in \KK[x_1,...,x_n] $ and $W$ the hypersurface
given by formal power  series $\F=\f+\sum s^j\g_j$, $g_j\in \KK[x_1,...,x_n]$.  It is not difficult to verify that the morphism $W\rightarrow S:=\spec
\KK[[s]]; (x_1,...,x_n,s)\mapsto s$  is flat.  We proved in \cite{Ley14} the following result:
 
\begin{theorem}[\cite{Ley14}]
\label{th:pham-bri}
Let $\f:=x_1^{a_1}+x_2^{a_2}+\cdots+x_2^{a_k}+\cdots+ x_n^{a_n}$, $a_k>1$,  and we suppose that the polynomials $\g_i$, $i\geq 1$, 
belong to the integral closure of the ideal generated by $x_1^{a_1},x_2^{a_2},\cdots, x_n^{a_n}$.  
Then the morphism $\mjp{m}  :  \red{(\mJet{W}{S}{m})}\rightarrow S:=\spec \KK[[s]]$  is flat for all $0\leq m\leq \infty$. 
In other words, $\red{(\mJet{W}{S}{m})}$ is a deformation of $\red{(V_m)}$.  
\end{theorem}

\subsection{Simultaneous Embedded Resolution} \label{sec:Hip-Gen}

Let $\KK$ be an algebraically closed field of characteristic zero, $A$ a local
$\KK$-algebra of finite type and
$V$ a hypersurface of $\AF_{\KK}^{n+1}$. We suppose that $W$ is an
{\it Embedded Deformation of $V$ over $S:=\spec A$}, 
that is we have the following commutative diagram:

\begin{center}

$\xymatrix @!0 @R=1.5pc @C=2pc{  
 & V \ar@{^{(}->}[rr] \ar@{->}[dd] && W \ar@{^{(}->}[rr] \ar@{->}[dd]^{\varrho} &&
 \AF_{S}^{n+1} \ar@{->}[lldd] &\hspace{3cm}:=\spec A[x_1,...,x_{n+1}]\\
  & & \square&  &&& \\
  0:= \hspace{0.3cm}&\spec \KK \ar@{^{(}->}[rr] && S&&&
}$   
\end{center}
where the morphism $\varrho$ is flat.\\

In the following we will define what we mean by {\it Simultaneous Embedded
Resolution} of an embedded deformation.  \\

We consider a proper birational morphism $\varphi: \widetilde{\AF}_{S}^{n+1}\rightarrow \AF_{S}^{n+1}$  
such that $\widetilde{\AF}_{S}^{n+1}$ is formally smooth over $S$, and we note for   $\widetilde{W}^{s}$ and
$\widetilde{W}^{t}$  the strict and total 
transform of $W$ in $\widetilde{\AF}_{S}^{n+1}$ respectively.
 We will say that the embedded deformation of $V$ over $S$ admits a  simultaneous embedded
resolution if the morphisms $\widetilde{W}^{s}\rightarrow W$ is a very weak simultaneous 
resolution and 
$\widetilde{W}^{t}$ is a  normal crossing divisor relative to $S$, that is to say, 
the induced morphism
$\widetilde{W}^{t}\rightarrow S$    is flat and for each $p\in \widetilde{W}^{t}$ there 
exists a Zariski open affine neighborhood $U\subset \tilde{\AF}_{S}^{n+1}$  of $p$, and
an \'etale morphism $\phi$.

 \begin{center}

$\xymatrix @!0 @R=1.5pc @C=2pc{  
 U \ar@{->}[rr]^{\phi} \ar@{->}[rdd]&& 
 \AF_S^{n+1}\ar@{->}[ldd]&\hspace{3cm}:=\spec A[y_1,...,y_{n+1}] \\
   &&& \\
   & S&& 
}$   
\end{center}
such that the $S$-scheme $\widetilde{W}^t\cap U$
is defined by the ideal $\phi^{\star}\mathcal{I}$, where $\mathcal{I}=(y_{1}^{a_{1}}\cdots y_{n+1}^{a_{n+1}})$,  $a_i\geq 0$.  
If $p\in \widetilde{W}^{s}$, we assume that $a_{n+1}=1$, and that $\widetilde{W}^s\cap U$ is defined by the ideal
$\phi^{\star}\mathcal{I}'$, 
where $\mathcal{I}'=(y_{n+1})$.\\

\subsection{${\bf \FN{\f}}$-Deformations} \label{se:new-def} Now we will show an important class 
of examples which satisfy  the hypothesis imposed above.\\

By abuse of notation we will note for $0$ the point of
$\AF^{n+1}_{\KK}:=\spec \KK[x_1,..,x_{n+1}]$  corresponding to the maximal ideal $(x_1,...,x_{n+1})$.  
Sometimes we will call this point the origin of $\AF^{n+1}_{\KK}$.\\

Let $V$ be a hypersurface of $\AF_{\KK}^{n+1}$ defined by an irreducible polynomial $\f=\sum c_ex^e \in \KK[x_1,...,x_{n+1}]$, where 
$x^e:=x_1^{e_1}\cdots x_{n+1}^{e_{n+1}}$ and $c_e\in \KK$, and $\mathcal{E}(\f):=\{e\in \ZZ_{\geq 0}^{n+1}\mid c_e\neq 0\}$.  
The Newton polyhedron  $\PN{\f}$ is the convex hull of the set $\{e\in \RR^{n+1}_{\geq 0}\mid e\in \mathcal{E}(\f)\}$, 
and the Newton boundary  $\FN{\f}$ of $\f$ is the union of compact faces of $\PN{\f}$. The Newton fan $\EN{\f}$
associated with $\f$ is the subdivision of the standard cone $\Delta:=\RR_{\geq 0}^{n+1}$, whose corresponding toric variety is the normalized blowing up of the ideal $\mathcal{I}(\f):=(\{x^e\mid e\in \FN{\f}\cap \ZZ^{N+1}\})$.\\

The polynomial $\f$ is non-degenerate with respect to the Newton boundary if for each compact face
$\gamma$ of $\PN{\f}$,
the polynomial $\f_{\gamma}:=\sum_{e\in \gamma}c_ex^e$  is non singular in the torus $T:=(\KK^{\star})^{n+1}$.\\    

We define the support function associated with $\PN{\f}$ as follows:

\begin{center}
 $\hf{\f}{p}=\inf\{\langle r,p\rangle \mid r\in \PN{\f}\}$, $p\in \Delta$.
\end{center}
where  $\langle\;,\;\rangle:\RR^{n+1}\times \RR^{n+1}\rightarrow \RR$  is the scalar product.\\

It is known that the Newton fan $\EN{\f}$  satisfies the following property (see \cite{Var76} or 
\cite{Mer80}):\\

$(\star)$ Let $J\subset \{1,...,n+1\}$ and  $\sigma_J:=\{(p_1,...,p_{n+1})\in  \Delta \mid p_i=0\; {\rm iff} \; i\neq J\}$.
If there exists $p\in \sigma_J$ such that $\hf{\f}{p}=0$,  then the adherence of $\sigma_J$ is a cone of  $\EN{\f}$.\\

A regular subdivision $\Sigma$ of  $\EN{\f}$ is called {\it admissible} if it satisfies the property ($\star$), in other words, 
if there exists  $p\in \sigma_J$ such that $\hf{\f}{p}=0$, then $\overline{\sigma_J}\in \Sigma$.  
The existence of the regular admissible subdivisions, under the assumption that $\f$ is non degenerate with respect to the Newton Boundary, 
was proved in \cite{Oka87}.\\

For the rest of this section we will assume that $\f$ is non degenerate with respect to the Newton Boundary, and that  $\REN{\f}$
is a regular admissible subdivision.  We suppose that the origin $0$ of  $\AF_{\KK}^{n+1}$ 
is the unique singular point of $V$, and that $V$ does not contain $T$-orbits of a strictly positive dimension.  
Then the toric morphism $\pi: X(\REN{\f})\rightarrow \AF_{\KK}^{n+1}$ is an embedded resolution of the singularity of the hypersurface $V$, 
which is to say $V^t_{red}$ is a simple normal crossing divisor.  
Observe that the irreducible components of $V^t_{red}$ are smooth varieties, in particular $V^s$ is a smooth variety.  For more details, see \cite{Var76}.\\

 It is worth noting that the property mentioned above only depends on the Newton polyhedron $\PN{\f}$,
 and on the fact that the polynomial $\f$ is non degenerate with respect to $\FN{\f}$, which is to say if $V'$ is a hypersurface, 
 with a unique singular point in $0$, given by the polynomial $\g$, non degenerate with respect to $\FN{\g}$ such that $\PN{\f}=\PN{\g}$,
 then $\pi: X(\REN{\f})\rightarrow \AF_{\KK}^{n+1}$  is an embedded resolution of the singularity $V'$.\\

Let $A$ be the localization of the ring of polynomials $\KK[s_1,...,s_l]$ at the maximal ideal  $(s_1,...,s_l)$. 
Consider the following polynomial.

 \begin{center}
  $\F:=\f(x_1,..,x_{n+1})+\sum_{j=1}^{l}s_j\g_j(x_1,...,x_{n+1})$
 \end{center}
 where  $\PN{\g_j}\subseteq \PN{\f}$ for all $j\in\{1,..,l\}$.  
 Let $W$ be the hypersurface of   $\AF_{S}^{n+1}$, $S:=\spec A$, defined by the polynomial $\F$.  Clearly $W$ is an embedded deformation of $V$ over $S$.  
	Let  $\Id:S\rightarrow S$ be the identity morphism  and let us consider the following  proper birational morphism:

 \begin{center}
 $\xymatrix @!0 @R=1.5pc @C=2pc{  
 \widetilde{\AF}_{S}^{n+1}:=X(\Sigma_{\f})\times S \ar@{->}[rrrrrr]^{\varphi:=\pi\times \Id} 
 &&&&&& \AF_{S}^{n+1}= \AF_{\KK}^{n+1}\times S
 \\}$   
\end{center}
 This morphism is a Simultaneous  embedded Resolution of $W$.

 \subsection{Deformation of  Spaces of $m$-jets}
 \label{sec:def-hyp} Let us remember that $\KK$ is an algebraically closed field of characteristic zero, 
 $A$ is  a local $\KK$-algebra of finite type, 
$V$ is a hypersurface of $\AF_{\KK}^{n+1}$, and  $\varrho:W\rightarrow S$ is an embedded deformation 
of $V$ over $S$. 
We consider a proper birational morphism $\varphi: \widetilde{\AF}_{S}^{n+1}\rightarrow \AF_{S}^{n+1}$  
such that $\widetilde{\AF}_{S}^{n+1}$ is formally smooth over $S$, and we note for   $\widetilde{W}^{s}$ and
$\widetilde{W}^{t}$  the strict and total 
transform of $W$ in $\widetilde{\AF}_{S}^{n+1}$ respectively. 
We note for $\varrho_m:\mJet{W}{S}{m}\rightarrow S$  the morphism induced by $\varrho$ (see Section \ref{sec:Hip-Gen}).\\

The following theorems are the main results of this article, which are the direct consequences
of Propositions \ref{pr:flat-mjet} and  \ref{pr:BW-flat}, which will be proved in this section.

 \begin{theorem}
 \label{th:flat-mjet}
  Let us suppose that $\varphi$ is a simultaneous embedded resolution of $W$, and that all embedded components of $\mJet{W}{S}{m}$ 
  dominate $S$.  Therefore for each $m\geq 0$, the morphism $\varrho_m:\mJet{W}{S}{m}\rightarrow S$ 
  is a deformation of the Space of $m$-jets $V_m$.   
  
 \end{theorem}
\begin{remark}
In the hypothesis of the theorem the condition that every embedded component 
of $\mJet{W}{S}{m}$ dominates S is satisfied,  for example  if $V_m$ does not have embedded components, or 
if  $\varphi_m:\mJet{\widetilde{W}^{t}}{S}{m}\rightarrow \mJet{W}{S}{m}$  
is schematically dominant. 
 \end{remark} 
\begin{theorem}
 \label{th:flat-mjet2}Let be $S=(\AF_{\KK},0)$ and let us suppose that $\varphi$ is a simultaneous embedded resolution of $W$.  
 Therefore for each $m\geq 0$ the morphism $\varrho_m:\red{(\mJet{W}{S}{m})}\rightarrow S$ 
is a deformation of $\red{(\mJet{W}{S}{m})}\times_{\AF_{\KK}}\{0\}$.  

\end{theorem}

\begin{remark}
 Observe that  $\red{(\red{(\mJet{W}{S}{m})}\times_{\AF_{\KK}}\{0\})}\cong \red{(V_m)}$.
\end{remark}

Initially we do not suppose that  $\varphi$ is a simultaneous embedded resolution of $W$

\begin{proposition}
\label{pr:morf-dom} 
The morphism $\varphi_m:\red{(\mJet{\widetilde{W}^{t}}{S}{m})}\rightarrow \red{(\mJet{W}{S}{m})}$ is surjective for  all  
$m\geq 0$.
\end{proposition} 

The following corollary follows directly from Proposition \ref{pr:morf-dom}.\\

We define the following application: 
\begin{center}
 $\sharp:\KK\!-\!\mathcal{S}ch\rightarrow \ZZ\cup{\infty}: V\mapsto \mbox{Number of irreducible components of $\red{V}$}$
\end{center}

\begin{corollary}
 For each $m \geq  0$, we have $\sharp(V_m)\leq \sharp{({V^t}_m})$. 
\end{corollary}
\begin{remark} 
If we  explicitly give a hypersurface $V$ and an embedded resolution of $V$, we can use the 
results of \cite{GoSm06} to bound the value of $\sharp{(V_m)}$.  
In general this corollary does not give optimal bounds, for example if the singularity of $V$ is rational, $\sharp{(V_m)}=1$, while 
$ \sharp{({V^t}_m})$ can be large. 

 \end{remark}

\begin{proof}[Proof of Proposition \ref{pr:morf-dom}]

 Let us consider the following commutative diagram of $S$-morphisms:

 \begin{center}

$\xymatrix @!0 @R=1.5pc @C=2pc{  
 \widetilde{W}^t \ar@{^{(}->}[rr] \ar@{->}[dd] && \widetilde{\AF}^{n+1}_{S} \ar@{->}[dd] \\
   & \commutatif& \\
 W \ar@{^{(}->}[rr]^{i} && \AF^{n+1}_{S}
}$   
\end{center}

This diagram naturally induces the following commutative diagram:

 \begin{center}

$\xymatrix @!0 @R=1.5pc @C=2pc{  
\red{(\mJet{\widetilde{W}^t}{S}{m})} \ar@{^{(}->}[rrrr] \ar@{->}[dd] &&&& \mJet{\widetilde{\AF}^{n+1}_{S}}{S}{m} 
\ar@{->}[dd]^{\varphi_m} \\
   && \commutatif & & \\
 \red{(\mJet{W}{S}{m})} \ar@{^{(}->}[rrrr] &&&& \mJet{\AF^{n+1}_{S}}{S}{m}}$   
\end{center}

 Let $\alpha$ be an  point of $\red{(\mJet{W}{S}{m})}$, and let  $\kappa(\alpha)$
be the residual field of the point $\alpha$.  The functorial property of the Spaces of  $m$-jets 
(see Theorem \ref{th:prop-fonct-mjet})  tells us that an $m$-jet corresponds to the point $\alpha$. 
By abuse of notation we will note this $m$-jet for $\alpha$.

\begin{center}
 $\alpha: \spec \kappa(\alpha)[t]/(t^{m+1})\rightarrow W$.
\end{center}

By abuse of notation, we note for $\alpha$ the $m$-jet  $i\circ \alpha:\spec \kappa(\alpha)[t]/(t^{m+1})\rightarrow \AF^{n+1}_{S}$.
As $\AF_{S}^{n+1}$ is formally smooth over $S$, the $S$-morphism

\begin{center}
 $\pinf{m}{}: \Arc{\AF^{n+1}_{S}}{S}\rightarrow \mJet{\AF^{n+1}_{S}}{S}{m}$
\end{center}
 is surjective. Let $\mathcal{I}(W)$ be the ideal of $W$.  As $W$ is a hypersurface, 
 there exists a polynomial $\F\in A[x_1,...,x_{n+1}]$ such that $\mathcal{I}(W)=(F)$.  The morphism $\pinf{m}{}$ is surjective, 
 then there exists an   arc $\beta:\spec \kappa(\alpha)[[t]]\rightarrow \AF_{S}^{n+1}$, such that $\pinf{m}{}(\beta)=\alpha$  and the ideal   
  $\beta^{\star} \mathcal{I}(W)$ satisfies that

\begin{center}
$\ord_t\beta^{\star} \mathcal{I}(W):=\min \{\ord_t r\mid r\in \beta^{\star} \mathcal{I}(W)\}  \geq m+1$.
\end{center}
observe that $\ord_t\beta^{\star} \mathcal{I}(W)=\ord_t \F(\beta^{\star}(t))$, where $\beta^{\star}(t)$ is the comorphism of $\beta$.\\ 

Let $Z\subset \AF_{S}^{n+1}$ be an $S$-scheme such that  $\varphi:\varphi^{-1}(U)\rightarrow U$ is an isomorphism, where $U:=\AF_{S}^{n+1}-Z$. 
Without loss of generality we can suppose that $\beta\not\in \Arc{Z}{S}$.  
As the morphism $\varphi:\widetilde{\AF}_{S}^{n+1}\rightarrow \AF_{S}^{n+1}$ is proper, there exists $\widetilde{\beta}\in \Arc{\widetilde{\AF}_{S}^{n+1}}{S}$ 
such that $\varphi\circ \widetilde{\beta}=\beta$.  Let $U'\subset \widetilde{\AF}^{n+1}_S $ be an  affine  open such that $\widetilde{\beta}(0)\in U'$,
then: 

\begin{center}
 $\ord \widetilde{\beta}^{\star}\mathcal{I}(U'\cap \widetilde{W}^t)\geq m+1$
\end{center}

Let us consider the following commutative diagram:

\begin{center}

$\xymatrix @!0 @R=1.5pc @C=2pc{  
\Arc{\widetilde{\AF}^{n+1}_S}{S} \ar@{^{(}->}[rrrr]^{\pinf{m}{\sim}} \ar@{->}[dd] &&&& \mJet{\widetilde{\AF}^{n+1}_{S}}{S}{m} \ar@{->}[dd] \\
   && \commutatif && \\
 \Arc{\AF^{n+1}_S}{S} \ar@{^{(}->}[rrrr]^{\pinf{m}{}} &&&& \mJet{\AF^{n+1}_{S}}{S}{m}
}$   
\end{center}

We define  $\widetilde{\alpha}:=\pinf{m}{\sim}(\widetilde{\beta})$.   By construction we obtain that $\varphi_{m}(\widetilde{\alpha})=\alpha$.\\  
As $\ord{\widetilde{\beta}^{\star}\mathcal{I}(U'\cap \widetilde{W}^t)}\geq m+1$, 
we obtain that  $\widetilde{\alpha}\in \red{(\mJet{\widetilde{W}^t}{S}{m})}$.   This ends the proof.   
\end{proof}

The following proposition is one of the key points of the principal resolution of this article.

\begin{proposition}
 \label{pr:flat-mjet} 
   Let us suppose that every embedded component of $\mJet{W}{S}{m}$ dominates $S$.  Let us fix $m\geq 0$. 
   If the morphism $\mJet{\widetilde{W}^t}{S}{m}\rightarrow S$  is flat then the morphism $\varrho_m:\mJet{W}{S}{m}\rightarrow S$ 
   is a deformation of the Space of $m$-jets $V_m$.

 \end{proposition}

 \begin{proof} For the proposition \ref{pr:mjets-chang-base} we have that $V_m\cong \mJet{W}{S}{m}\times_{\AF_{\KK}}\{0\}$. 
  So it suffices to prove that the morphism   $\mJet{W}{S}{m}\rightarrow S$ is flat.\\

The flatness is a local property (see, for example, Proposition $2.13$ from page $10$ of \cite{Liu02},
or Proposition $9.1A$ from page $253$ of \cite{Har77}), then it is enough to consider $p \in V_{m}$ 
arbitrary, and to prove that the germ of a neighborhood of $p$, $ (\mJet{W}{S}{m},p)$, is flat over S.\\ 
 
 Let us suppose that $p\in \mJet{W}{S}{m}$.  By virtue of Proposition \ref{pr:morf-dom}
 we can choose $q\in \varphi_m(p)^{-1}$ such that the morphisms of germs induced by $\varphi_m$ is dominant.

   \begin{center}

$\xymatrix @!0 @R=1.5pc @C=2pc{  
 &(\red{(\mJet{\widetilde{W}^t}{S}{m})},q) \ar@{->}[rrrr]^{\varphi_m} \ar@{->}[rrdd]&&&&
 (\red{(\mJet{W}{S}{m})},p)\ar@{->}[lldd] \\
(\star)\mbox{\hspace{2cm}}  & &&&&\\
   &&& S && 
}$
\end{center}

  The following lemma will be very useful in this proof.

\begin{lemma}
\label{le:mjet-domi}
  Every associated component of $(\mJet{W}{S}{m},p)$  dominates   $S$.  
 \end{lemma}
\begin{proof}
For the  hypotheses of the proposition  we are only interested in the associated components that are not embedded components.  
Then it is enough to prove that all components of $(\red{(\mJet{W}{S}{m})},p)$ dominate $S$. 

The morphism$(\mJet{\widetilde{W}^t}{S}{m}, q)\rightarrow S$
is flat, which implies that each irreducible component of $(\red{(\mJet{\widetilde{W}^t}{S}{m})}, q)$ dominates $S$ (see Lemma $3.7$ from page $136$ of \cite{Liu02}).
 Using the fact that the morphism  $\varphi_m:\red{(\mJet{\widetilde{W}^t}{S}{m})} \rightarrow \red{(\mJet{\widetilde{W}}{S}{m})}$ 
    is dominant (Proposition \ref{pr:morf-dom}), and the commutative diagram  $(\star)$, we obtain that each  irreducible component 
  of $(\red{(\mJet{W}{S}{m})},p)$ dominates $S$.    
\end{proof}

We will begin with the case $S:=(\AF_{\KK}, 0)$.  The condition of flatness over 
$(\AF_{\KK}, 0)$ is equivalent to  each associated components of 
 $(\mJet{W}{S}{m},p)$ dominating $(\AF_{\KK}, 0)$ (see Proposition $9.7$ from page $257$ of \cite{Har77}). 
 Which is obtained directly from the previous lemma.\\

Now we will suppose that the theorem is true for all schemes $(\AF_{\KK}^{l_0}, 0)$ with $l_0\leq l-1$.\\

Let $S:=(\AF_{\KK}^{l}, 0)$.  To demonstrate this case we will use the Corollary $6.9$ on page $170$ of \cite{Eis95}.
This is to say that we will prove that exists $s\in \Gamma(\mathcal{O}_S)$    such that $s$ 
is not a divisor of the zero in $\Gamma(\mathcal{O})$ and that
$\Gamma(\mathcal{O})/s\Gamma(\mathcal{O})$ 
is flat over  $\Gamma(\mathcal{O}_S)/(s)$, 
where $\mathcal{O}:=\mathcal{O}_{(\mJet{W}{S}{m},p)}$.\\

Let us consider the injective morphism $\KK[s_1,...,s_{l}]_m\rightarrow \Gamma(\mathcal{O})$,
where $m:=(s_1,...,s_{l})$.  Let us suppose that there exists $h\in \Gamma(\mathcal{O})- \{0\}$
such that $s_{1}h=0$. Then the associated ideal $\rm{ass}(s_1)\neq 0$  contradicting Lemma \ref{le:mjet-domi}.

 Then  $s_1$  is not a divisor of the zero in  
$\Gamma(\mathcal{O})$.\\

Now let us consider the following commutative diagram:

\begin{center}

$\xymatrix @!0 @R=1.5pc @C=2pc{  
 (\mJet{W}{S}{m},p)\times_{S}S'\ar@{^{(}->}[rrrrrr] \ar@{->}[dd] &&&&&& (\mJet{W}{S}{m},p)  \ar@{->}[dd]^{\varrho_m} \\
  & && & &&\\
    S':=\spec\KK[s_2,...,s_{l}]_{m'}\ar@{^{(}->}[rrrrrr] &&&&&& S:=\spec \KK[s_1,...,s_{l}]_m
}$   
\end{center}
where $m':=(s_2,...,s_{l})\subset \KK[s_2,...,s_{l}]$. 

By virtue of  Proposition \ref{pr:mjets-chang-base}, 
we obtain that  $(\mJet{W}{S}{m},p)\times_{S}S'\cong (\mJet{W}{S'}{m},p)$.  Now using the 
induction hypothesis we obtain that $(\mJet{W}{S}{m},p)\times_{S}S'\rightarrow S'$ is a flat morphism. 
Thereby we obtain that the morphism $\mJet{W}{S}{m}\rightarrow S$ is flat.\\ 

Let us consider $S$ a local $\KK$ scheme of the finite type.  
Then there exists an integer $l$ and a morphism $S\rightarrow (\AF^{l},0)$.  
Using base extension, 
the Proposition $9.1A$ of \cite{Har77}, and the Proposition \ref{pr:mjets-chang-base},
we obtain that the morphism $\mJet{W}{S}{m}\rightarrow S$ is flat.    \end{proof} 

From this point on we will assume that the morphism $\varphi$ is a Simultaneous Embedded Resolution 
of $W$. 
 
 \begin{proposition}
 \label{pr:BW-flat}
  The morphism $\mJet{\widetilde{W}^t}{S}{m}\rightarrow S$ is flat for all $m\geq 0$.
 \end{proposition}
\begin{proof}
The flatness is a local property, then for each  point $p\in \widetilde{W}^t$ it suffices to consider an open $U\subset \widetilde{\AF}_{S}^{n+1}$
and to prove that the morphism $\mJet{\widetilde{W}^t_{U}}{S}{m}\rightarrow S$ is flat, where $\widetilde{W}^t_{U}:=U\cap\widetilde{W}^t_{U}$.\\

 Let us consider the open affine $U\subset \widetilde{\AF}_{S}^{n+1}$ from section  \ref{sec:Hip-Gen}, in other words, 
 let us suppose that there exists an \'etale morphism

 \begin{center}

$\xymatrix @!0 @R=1.5pc @C=2pc{  
 U \ar@{->}[rr]^{\phi} \ar@{->}[rdd]&& 
 \AF_S^{n+1}\ar@{->}[ldd]&\hspace{3cm}:=\spec A[y_1,...,y_{n+1}] \\
   &&& \\
   & S&& 
}$   
\end{center}
such that the $S$-scheme  $\widetilde{W}^t_U$ is defined by the ideal $\phi^{\star}\mathcal{I}$, where
$\mathcal{I}=(y_{1}^{a_{1}}\cdots y_{n+1}^{a_{n+1}})$,  $a_i\geq 0$.\\  

Let $Y$ be the $S$-scheme defined by the ideal $\mathcal{I}$.  Then 

\begin{center}
 
$\Gamma(Y,\mathcal{O}_{\mJet{Y}{S}{m}}):=A[
\{y_{i j}\mid 1\leq i\leq n+1,\; 0\leq j\leq m\}]/J_m$,
\end{center}
where $J_m$ is generated by the polynomials $\frac{1}{i!}\partial_t^{i}y(t)\mid_{t=0}$, $0\leq i\leq m$, where 
\begin{center}

$y(t)=(y_{10}+y_{11}t\cdots)^{a_1}\cdots(y_{n+1 0}+y_{n+1 1}t\cdots)^{a_{n+1}}$.
\end{center}

Because the generators of $J_m$ are not  dependent on $A$, we obtain that the morphism $\mJet{Y}{S}{m}\rightarrow S$ is flat.  
Because the morphism $\phi$  is \'etale, there exists an isomorphism of $\widetilde{W}^t_U$-schemes:  

\begin{center}
 $\mJet{\widetilde{W}^t_U}{S}{m}\cong \mJet{Y}{S}{m}\times_{Y} \widetilde{W}^t_U$
\end{center}
By base change the morphism $\widetilde{W}^{t}_U\rightarrow Y$ is \'etale, in particular it is flat, which implies that  $\mJet{\widetilde{W}^t_U}{S}{m}$
is flat over $S$. 
\end{proof}

The following result is an application of Theorem \ref{th:flat-mjet}.\\

Let $v:=(v_1,...,v_n)\in \ZZ^n_{>0}$, $d\in \ZZ_{>0}$ et let $\f\in \KK[x_1,...,x_n]$ be a quasi-homogeneous
polynomial of the type $(d,v)$, which is to say that $\f$ is homogeneous of degree $d$ 
with respect to the graduation   $\nu_v x_i=v_i$.  Let $\h$ belong to  $\KK[x_1,...,x_n]$ such that $\nu_v\h>d$.  
We note $V$ (resp. $V'$) the hypersurface given by the equation  $\f(x_1,...,x_n)=0$ (resp. $\f(x_1,...,x_n)+\h(x_1,...,x_n)=0$).
We suppose that $V$ (resp. $V'$) has a unique isolated singularity at the origin of $\AF_{\KK}^n$, 
and that $V$ (resp. $V'$) does not contain any of $T$-orbit of $\AF_{\KK}^n$. In addition we suppose that 
$\f$ and  $\f_1:=\f+\h$ are non degenerate  with respect to the Newton boundary $\Gamma(\f)$.
We remark that the morphism  $\pi:X(\Sigma_{\f})\rightarrow \AF^{n}$, where $\Sigma_{\f}$ is an admissible regular fan,  
is an embedded resolution of varieties $V$ and $V'$, and that the exceptional fibers of the desingularizations $\pi_0:X\rightarrow V$,  $\pi_1:X'\rightarrow V'$ 
induced by the embedded  resolution  $\pi:X(\Sigma_{\f})\rightarrow \AF^{n}$ are trivially homeomorphic.  
In the following theorem, by abuse of notation, we use the same notation to designate the 
irreducible components of the exceptional fibers of $\pi_0$ and $\pi_1$. We can apply the following theorem 
(proved in \cite{Ley14}), to our new hypothesis:

\begin{theorem}[\cite{Ley14}]
 \label{th:def-appl-nash}
  If the divisor $\E$ belongs to the image of the Nash Application $\mathcal{N}_{V}$, then $\E$
  belongs to the Nash Application  $\mathcal{N}_{V'}$. 
 \end{theorem}

\end{document}